\documentclass[12pt]{article}
\usepackage{graphicx,epsfig}
\usepackage{color}
\usepackage{amsmath,amssymb,mathrsfs,amsthm}
\usepackage{amsfonts}
\usepackage{multicol}
 \usepackage{graphicx}
\usepackage{color}
\usepackage{pdfcolmk}
\usepackage[nice]{nicefrac}
\usepackage{amsthm}
\usepackage{latexsym}
\usepackage{dsfont}
\usepackage{setspace}
\usepackage[a4paper, tmargin=1.5in, bmargin=1in, lmargin=1in, rmargin=1in, headheight=13.6pt]{geometry}
\usepackage[sort&compress,numbers]{natbib}
\usepackage[font=small,labelfont=bf,labelsep=space]{caption}
\makeatletter

\date{}

\def\@citex[#1]#2{\if@filesw\immediate\write\@auxout{\string\citation{#2}}\fi
  \def\@citea{}\@cite{\@for\@citeb:=#2\do
    {\@citea\def\@citea{,\linebreak[0]\hskip0pt plus .2em}%
      \@ifundefined{b@\@citeb}%
    {{\bf ?}\@warning{Citation `\@citeb' on page \thepage\space undefined}}%
      \hbox{\csname b@\@citeb\endcsname}}}{#1}}

\newtheorem{theorem}{Theorem}[section]

\newtheorem{remark}{Remark}[section]

\newtheorem{example}{Example}[section]

\newtheorem{rule-def}[theorem]{Rule}
\numberwithin{equation}{section}
\makeatletter

\begin{document}
\title{Analytic solution of system of singular nonlinear  differential equations with Neumann-Robin boundary conditions arising in
  astrophysics}
\author{Randhir Singh\thanks{Corresponding author. E-mail:~{randhir.math@gmail.com,} } \\
\small $^{*}$ Department of Mathematics, Birla Institute of Technology Mesra, Ranchi-835215, India.
}\maketitle{}
\begin{abstract}
\noindent In this paper, we propose a new approach for the approximate analytic solution of system of Lane-Emden-Fowler type equations with Neumann-Robin boundary conditions. The algorithm is based on Green's function and the homotopy analysis method. This approach depends on constructing Green's function before establishing the recursive scheme for the approximate analytic solution of the equivalent system of integral equations. Unlike Adomian decomposition method (ADM)  \cite{singh2020solving}, the present method contains adjustable parameters to control the convergence of the approximate series solution. Convergence and error estimation of the present  is provided under quite general conditions.
Several examples are considered to demonstrate the accuracy of the current algorithm. Computational results reveal that the proposed approach produces better results as compared to some existing iterative methods.
\end{abstract}
\section{Introduction}

We consider the following system of singular boundary value problems (SBVPs)
\begin{equation}\label{sec1:eq1}
\left\{
  \begin{array}{ll}
\displaystyle \big(p_{1}(x)y'_{1}(x)\big)'=p_1(x)f_1\big(x,y_1(x),y_2(x)\big),~~~~~~~~x\in(0,1), \vspace{0.1cm}\\
\displaystyle \big(p_{2}(x)y'_{2}(x)\big)'=p_2(x)f_2\big(x,y_1(x),y_2(x)\big),\vspace{0.1cm}\\
y'_{1}(0)=0,~~a_{1} y_{1}(1)+b_{1} y_{1}'(1)=c_{1},~~~y'_{2}(0)=0,~~a_{2}  y_{2}(1)+b_{2} y_{2}'(1)=c_{2},
\end{array} \right.
\end{equation}
where $a_1, a_2, b_1, b_2,c_1, c_2$ are real constants. Here,  $p_1(x)=x^{k_1}g_1(x)$,   $p_2(x)=x^{k_2}g_2(x)$, $g_1(0)\neq0$ and  $g_2(0)\neq0$ with $p_1(0)=p_2(0)=0$.

We next consider the following system of Lane-Emden equations \cite{muatjetjeja2010noether,flockerzi2011coupled,rach2014solving,wazwaz2014study,singha2019efficient,hao2018efficient,singh2020solving} a particular case of \eqref{sec1:eq1} with $p_1(x)=x^{k_1}$  and $p_2(x)=x^{k_2}$ as
\begin{equation}\label{sec1:eq2}
\left\{
  \begin{array}{ll}
\displaystyle (x^{k_1}y'_{1}(x))' =x^{k_1}f_1\big(x,y_1(x),y_2(x)\big),~~~~~~x\in(0,1),\vspace{0.1cm}\\
\displaystyle (x^{k_2}y'_{2}(x))' =x^{k_2}f_2\big(x,y_1(x),y_2(x)\big),\vspace{0.10cm}\\
y'_{1}(0)=0,~~a_{1} \ y_{1}(1)+b_{1}\ y_{1}'(1)=c_{1},~~~y'_{2}(0)=0,~~a_{2} \ y_{2}(1)+b_{2}\ y_{2}'(1)=c_{2}.
\end{array} \right.
\end{equation}
In recent years, the SBVPs  have been studied extensively in  \cite{wazwaz2011comparison,inc2005different,mittal2008solution,wazwaz2011variational,khuri2010novel,ebaid2011new,wazwaz2013adomian,singh2013numerical,
singh2014adomian,singh2014efficient,singh2014approximate,das2016algorithm,singh2016efficient,singh2016effective,singh2017optimal,verma2018convergence,verma2019higher,
singh2019haar,singh2019Ahaar,singh2019modified,singh2020haar,shahni2020efficient,kumar2020numerical} and references therein.  But as far as we know, we find only the following results on system of Lane-Emden equations.  Recently, in \cite{muthukumar2014analytical,wazwaz2014study,duan2015oxygen} authors studied  \eqref{sec1:eq2} with  boundary conditions $y'_{1}(0)=y'_{2}(0)=0,~y_{1}(1)=y_{2}(1)=1$ and shape factors $k_1=k_2=2$ and
$f_1=-b+a\ g_1\big(y_1,y_2\big)+c\ g_2(y_1,y_2),~~f_2=d\ g_1(y_1,y_2)+e\ g_2(y_1,y_2),$ that relates the concentration of the carbon substrate and the concentration of oxygen,
the Michaelis-Menten functions $$g_i(y_1,y_2)=M_{1} \times M_{2},~~~~~M_{1}=\frac{y_1}{l_{i}+y_1}, ~~~M_{2}=\frac{y_2}{m_{i}+y_2},~ ~~i=1,2,$$ where $M_{1}$ and  $M_{2}$ are the respective Michaelis-Menten nonlinear operators.

In \cite{flockerzi2011coupled,rach2014solving,madduri2019fast}, authors considered the  system of Lane--Emden equations \eqref{sec1:eq2} with  boundary conditions $y'_{1}(0)=y'_{2}(0)=0,~y_{1}(1)=1,~y_{2}(1)=2$ and shape factors $k_1=k_2=2$ and $f_1=a\ y_{1}^{2}+b\ y_1 y_2,~~f_2=c\ y_{1}^{2}-d\ y_1  y_2,$
occurs in catalytic diffusion reactions with the parameters $a, b, c$, $d$ are actual chemical reactions.

In \cite{rach2014solving,duan2015oxygen}, the ADM was applied to obtain a convergent analytic approximate solution  of \eqref{sec1:eq2} with $k_1=k_2=2$. Later, in \cite{wazwaz2016variational}, the variational iteration method (VIM) was applied to obtain approximations to solutions of \eqref{sec1:eq2} for shape factors $k_i=1,2,3,~i=1,2$.  Most recently, a numerical procedure  based on sinc-collocation method was developed in \cite{saadatmandi2018numerical}  to obtain the solution of  \eqref{sec1:eq2}. In \cite{hao2015solving} authors used the reproducing kernel Hilbert space method for solving  to obtain the solution of  \eqref{sec1:eq2}. In \cite{singh2020solving}, the ADM with Green's function used to find numerical approximation of the solutions of \eqref{sec1:eq2}.

The homotopy analysis method was developed and improved by S. Liao \cite{liao2003beyond,liao2004homotopy,liao2007general,liao2009series} for solving a broad class of functional equations. Various modifications of the homotopy analysis method have also been elaborated, for example, the optimal homotopy asymptotic method was proposed by Marinca and  Herisanu \cite{marinca2008application,marinca2008optimal}, the optimal homotopy analysis method was introduced in \cite{liao2010optimal,liao2012homotopy}. Recently, the homotopy analysis method with Green's function was applied to solve singular boundary value problems in  \cite{singh2018optimal,singh2018Oxygen,singha2019optimal,singh2019analytic,singh2018analytical}.

In this work, we present the homotopy analysis method combined with the Green's function strategy for obtaining approximate solutions of coupled singular boundary value problems \eqref{sec1:eq1} and  \eqref{sec1:eq2}.  Unlike standard HAM, our approach avoids solving the transcendental equations for the undetermined coefficients. Unlike the ADM and VIM, our proposed technique contains the convergence-control parameters which gives fast convergence of the series solution. Numerical results reveal that the present process provides better results as compared to some existing iterative methods \cite{rach2014solving,wazwaz2014study,singh2020solving}.

\section{Integral form of system of  Lane-Emden-Fowler types   equations}
By following \cite{singh2016efficient}, we transform the following system of Lane-Emden-Fowler types   equations with Neumann-Robin boundary conditions
\begin{equation}\label{sec2:e1}
\left\{
  \begin{array}{ll}
\displaystyle \big(p_{1}(x)y'_{1}(x)\big)'=p_1(x)f_1\big(x,y_1(x),y_2(x)\big),~~~~~~~~x\in(0,1), \vspace{0.1cm}\\
\displaystyle \big(p_{2}(x)x'_{2}(x)\big)'=p_2(x)f_2\big(x,y_1(x),y_2(x)\big),\vspace{0.1cm}\\
y'_{1}(0)=0,~~a_{1}  y_{1}(1)+b_{1} y_{1}'(1)=c_{1},~~~y'_{2}(0)=0,~~a_{2}  y_{2}(1)+b_{2} y_{2}'(1)=c_{2},
\end{array} \right.
\end{equation}
into the equivalent system of integral equations
\begin{align}\label{sec2:e2}
\left\{
  \begin{array}{ll}
\displaystyle y_{1}(x)=\frac{c_1}{a_1}+\int\limits_{0}^{1}G_{1}(x,s)\ p_1(s)\ f_1\big(s,y_1(s),y_2(s)\big) ds,\\
\displaystyle y_{2}(x)=\frac{c_2}{a_2}+\int\limits_{0}^{1}G_{2}(x,s)\ p_2(s)\ f_2\big(s,y_1(s),y_2(s)\big) ds,
\end{array} \right.
\end{align}
where  $G_{1}(x,s)$ and $G_{2}(x,s)$ are given by
\begin{align}\label{sec2:e3}
G_{1}(x,s)=\left   \{
  \begin{array}{ll}
\displaystyle h_1(1)-h_1(x) +\frac{b_1}{a_1}h'_1(1), & \hbox{$s\leq x$},\vspace{.25cm}\\
\displaystyle h_1(1)-h_1(s)+\frac{b_1}{a_1}h'_1(1), & \hbox{$x \leq s$},
   \end{array}
\right.
\end{align}
with $h_1(x)=\int\limits_{0}^{x}\frac{dt}{p_1(t)},~h_1'(1)=\frac{1}{p_1(1)},~h_1(1)-h_1(s)=\int\limits_{s}^{1}\frac{dt}{p_1(t)},~
h_1(1)-h_1(x)=\int\limits_{s}^{1}\frac{dt}{p_1(t)}-\int\limits_{s}^{x}\frac{dt}{p_1(t)}$,
\begin{align}\label{sec2:e4}
G_{2}(x,s)=\left   \{
  \begin{array}{ll}
\displaystyle h_2(1)-h_2(x) +\frac{b_2}{a_2}h'_2(1), & \hbox{$s\leq x$},\vspace{.2cm} \\
\displaystyle h_2(1)-h_2(s)+\frac{b_2}{a_2}h'_2(1), & \hbox{$x\leq s$},
   \end{array}
\right.
\end{align}
with $h_2(x)=\int\limits_{0}^{x}\frac{dt}{p_2(t)},~h_2'(1)=\frac{1}{p_2(1)},~~h_2(1)-h_2(s)=\int\limits_{s}^{1}\frac{dt}{p_2(t)},~h_2(1)-h_2(x)=\int\limits_{s}^{1}\frac{dt}{p_2(t)}-\int\limits_{s}^{x}\frac{dt}{p_2(t)}.
$

Similarly, by following \cite{singh2014efficient,singh2020solving}, we transform the following system of Lane-Emden types   equations with Neumann-Robin boundary conditions
\begin{equation}\label{sec2:e5}
\left\{
  \begin{array}{ll}
 \displaystyle (x^{k_1}y'_{1}(x))' =x^{k_1}f_1\big(x,y_1(x),y_2(x)\big),~~~~~~x\in(0,1),\vspace{0.1cm}\\
\displaystyle (x^{k_2}y'_{2}(x))' =x^{k_2}f_2\big(x,y_1(x),y_2(x)\big),\vspace{0.10cm}\\
y'_{1}(0)=0,~~a_{1} \ y_{1}(1)+b_{1}\ y_{1}'(1)=c_{1},~~~y'_{2}(0)=0,~~a_{2} \ y_{2}(1)+b_{2}\ y_{2}'(1)=c_{2},
\end{array} \right.
\end{equation}
into the equivalent system of integral equations
\begin{align}\label{sec2:e6}
\left\{
  \begin{array}{ll}
 \displaystyle y_{1}(t)=\frac{c_1}{a_1}+\int\limits_{0}^{1}G_{1}(x,s)\ s^{k_1}\ f_1\big(s,y_1(s),y_2(s)\big) ds,\\
 \displaystyle y_{2}(x)=\frac{c_2}{a_2}+\int\limits_{0}^{1}G_{2}(x,s)\ s^{k_2}\ f_2\big(s,y_1(s),y_2(s)\big) ds,
\end{array} \right.
\end{align}
where $G_{1}(x,s)$ and $G_{2}(x,s)$ for $k_1=k_2=1$  are
\begin{align}\label{sec2:e7}
G_{1}(x,s)= \left   \{
  \begin{array}{ll}
  \ln x-\displaystyle\frac{b_1}{a_1}, & \hbox{$s\leq x$},\vspace{.2cm}\\
  \ln s-\displaystyle\frac{b_1}{a_1}, & \hbox{$x \leq s$}, \\
\end{array}
\right.\\
G_{2}(x,s)= \left   \{
  \begin{array}{ll}
  \ln x-\displaystyle\frac{b_2}{a_2}, & \hbox{$s\leq x$},\vspace{.2cm}\\
  \ln s-\displaystyle\frac{b_2}{a_2}, & \hbox{$x \leq s$}
\end{array}
\right.
\end{align}
and $G_{1}(x,s)$ and $G_{2}(x,s)$ for $k_1\neq 1,~~k_2\neq 1$  are
\begin{align}\label{sec2:e8}
G_{1}(x,s)= \left   \{
  \begin{array}{ll}
  \displaystyle \frac{x^{1-k_1}-1}{1-k_1}-\frac{b_1}{a_1}, & \hbox{$s\leq x$}\vspace{.15cm} \\
   \displaystyle \frac{s^{1-k_1}-1}{1-k_1}-\frac{b_1}{a_1}, & \hbox{$x \leq s$},
  \end{array}
\right.\\
G_{2}(x,s)= \left   \{
  \begin{array}{ll}
  \displaystyle \frac{x^{1-k_2}-1}{1-k_2}-\frac{b_2}{a_2}, & \hbox{$s\leq x$},\vspace{.15cm} \\
   \displaystyle \frac{s^{1-k_2}-1}{1-k_2}-\frac{b_2}{a_2}, & \hbox{$x \leq s$}.
  \end{array}
\right.
\end{align}

\section{Solution method}
We consider the system of integral operator equation \eqref{sec2:e2} as
\begin{align}\label{sec2:e4a}
\left\{
  \begin{array}{ll}
 N_1[y_{1}]=\displaystyle y_{1}(x)-\frac{c_1}{a_1}-\int\limits_{0}^{1}G_{1}(x,s)\ p_1(s)\ f_1\big(s,y_1(s),y_2(s)\big) ds=0,\\
 N_2[y_{2}]=\displaystyle  y_{2}(x)-\frac{c_2}{a_2}-\int\limits_{0}^{1}G_{2}(x,s)\ p_2(s)\ f_2\big(s,y_1(s),y_2(s)\big) ds=0.
 \end{array} \right.
\end{align}
Note that the above system of integral equations reduce to \eqref{sec2:e6} with  $p_1(x)=x^{k_1}$  and $p_2(x)=x^{k_2}$.

In order to apply the HAM with Green's function \cite{singh2018optimal,singh2018Oxygen,singha2019optimal,singh2019analytic,singh2018analytical}, we construct the zeroth-order deformation equations for \eqref{sec2:e4a}  as
\begin{align}\label{sec2:e4}
\left\{
  \begin{array}{ll}
(1-q)[\varphi_1(x,q)-y_{10}(x)]=q \ c_{10} \ N_1[\varphi_1(x,q)],\vspace{.2cm} \\
(1-q)[\varphi_2(x,q)-y_{20}(x)]=q \ c_{20} \ N_2[\varphi_2(x,q)],~~~q\in[0,1]
 \end{array} \right.
\end{align}
where  $q$ is an embedding parameter,    $(y_{10}(x),y_{20}(x))$ are initial approximations,  $(c_{10}\neq0, c_{20}\neq0)$ are convergence control parameters and   ($\varphi_1(x,q), \varphi_2(x,q)$) are  unknown functions.

If $q=0$  and $q = 1$ the homotopy equations \eqref{sec2:e4}  vary from $\big(\varphi_1(x,0), \varphi_2(x,0)\big)=\big(y_{10}(x),y_{20}(x)\big) \rightarrow \big(\varphi_1(x,1), \varphi_2(x,1)\big)=\big(y_{1}(x),y_{2}(x)\big)$. Now as $q$ varies from $0$ to $1$, the solutions of \eqref{sec2:e4a} will vary from the initial guesses $\big(y_{10}(x),y_{20}(x)\big)$ to the exact solutions  $\big(y_{1}(x),y_{2}(x)\big)$ of equations \eqref{sec2:e4a}.

Expanding $\varphi_1(x,q)$  and $\varphi_2(x,q)$ as a Taylor series with respect to $q$ yields
\begin{align}\label{sec2:e7}
\left\{
  \begin{array}{ll}
\displaystyle  \varphi_1(x,q)=\sum_{k=0}^{\infty} y_{1k}(x) q^{k},\vspace{.2cm} \\
\displaystyle  \varphi_2(x,q)=\sum_{k=0}^{\infty} y_{2k}(x) q^{k},
 \end{array} \right.
\end{align}
where
\begin{align}\label{sec2:e8}
y_{1k}(x)=\frac{1}{k!}\frac{\partial^k \varphi_{1}(x,q)}{\partial q^k}\bigg|_{q=0},~~~y_{2k}(x)=\frac{1}{k!}\frac{\partial^k \varphi_{2}(x,q)}{\partial q^k}\bigg|_{q=0}.
\end{align}
The series \eqref{sec2:e7}  converge for $q = 1$ if  $c_{10}\neq0,~c_{20}\neq0$ are chosen properly, and they take the form
\begin{align}\label{sec2:e9}
\left\{
  \begin{array}{ll}
\displaystyle \varphi_{1}(x,1)\equiv y_{1}(x)=\sum_{k=0}^{\infty} y_{1k}(x),\vspace{.2cm} \\
\displaystyle \varphi_{2}(x,1)\equiv y_{2}(x)=\sum_{k=0}^{\infty} y_{2k}(x),
 \end{array} \right.
\end{align}
which will be the solutions of  \eqref{sec2:e4a}.

For further analysis, we define the vectors as
\begin{align}
\left\{
  \begin{array}{ll}
\displaystyle \overrightarrow{\textbf{y}}_{1k} = \{y_{10}(x), y_{11}(x), \ldots, y_{1k}(x)\}, \vspace{.2cm} \\
\displaystyle \overrightarrow{\textbf{y}}_{2k} = \{y_{20}(x), y_{21}(x), \ldots, y_{2k}(x)\}.
 \end{array} \right.
\end{align}
Differentiating \eqref{sec2:e4}  $k$-times with respect to $q$, dividing them by $k!$,
setting $q = 0,$ we obtain  the $k$th-order deformation equations as
\begin{align}\label{sec2:e10}
\left\{
  \begin{array}{ll}
\displaystyle y_{1k}(x)-\chi_{k} \ y_{1(k-1)}(x)=c_{10}\ R_{1k}[\overrightarrow{\textbf{y}}_{1(k-1)}], \vspace{.2cm}\\
\displaystyle y_{2k}(x)-\chi_{k} \ y_{2(k-1)}(x)=c_{20}\ R_{2k}[\overrightarrow{\textbf{y}}_{2(k-1)}],
 \end{array} \right.
\end{align}
where $\chi_{k}=\left   \{
  \begin{array}{ll}
     0, & k\leq1, \\
     1, & k>1,
\end{array}
\right. $
and
\begin{align*}
\left\{
  \begin{array}{ll}
\displaystyle R_{1k}[\overrightarrow{\textbf{y}}_{1(k-1)}]=\frac{1}{(k-1)!} \frac{\partial^{k-1}}{\partial q^{k-1}} \bigg[\varphi_{1}(x,q)-\frac{c_1}{a_1}-\int\limits_{0}^{1}G_{1}(x,s) p_1(s)f_1\big(s,\varphi_{1}(s,q),\varphi_{2}(s,q)\big) ds\bigg] \bigg|_{q=0} \vspace{.2cm} \\
\displaystyle R_{2k}[\overrightarrow{\textbf{x}}_{2(k-1)}]=\frac{1}{(k-1)!} \frac{\partial^{k-1}}{\partial q^{k-1}} \bigg[\varphi_{2}(x,q)-\frac{c_2}{a_2}-\int\limits_{0}^{1}G_{2}(x,s) p_2(s)f_2\big(s,\varphi_{1}(s,q),\varphi_{2}(s,q)\big) ds\bigg] \bigg|_{q=0}
 \end{array} \right.
\end{align*}
\begin{align*}
\left\{
  \begin{array}{ll}
\displaystyle  =\frac{1}{(k-1)!} \frac{\partial^{k-1}}{\partial q^{k-1}} \bigg[\sum_{i=0}^{\infty} y_{1i} q^{i}-\frac{c_1}{a_1}-\int\limits_{0}^{1}G_{1}(x,s) p_1(s)f_1\bigg(s,\sum_{i=0}^{\infty} y_{1i} q^{i}, \sum_{i=0}^{\infty} y_{2i} q^{i} \bigg) ds\bigg] \bigg|_{q=0}, \vspace{.2cm}\\
\displaystyle  =\frac{1}{(k-1)!} \frac{\partial^{k-1}}{\partial q^{k-1}} \bigg[\sum_{i=0}^{\infty} y_{2i} q^{i}-\frac{c_2}{a_2}-\int\limits_{0}^{1}G_{2}(x,s)\ p_2(s)f_2\bigg(s,\sum_{i=0}^{\infty} y_{1i} q^{i},\sum_{i=0}^{\infty} y_{2i}\ q^{i}\bigg) ds\bigg] \bigg|_{q=0}.
 \end{array} \right.
\end{align*}
On simplification we get
\begin{align}\label{sec2:e12}
\left\{
  \begin{array}{ll}
\displaystyle   R_{1k}[\overrightarrow{\textbf{y}}_{1(k-1)}]=y_{1(k-1)}-\bigg(1-\chi_{k}\bigg)\frac{c_1}{a_1}-\int\limits_{0}^{1}G_1(x,s) p_1(s) H_{1(k-1)}ds, \vspace{.2cm} \\
 \displaystyle   R_{2k}[\overrightarrow{\textbf{y}}_{2(k-1)}]=y_{2(k-1)}-\bigg(1-\chi_{k}\bigg)\frac{c_2}{a_2}-\int\limits_{0}^{1}G_2(x,s) p_2(s) H_{2(k-1)}ds,
 \end{array} \right.
\end{align}
where $H_{1k}$ and $H_{2k}$ are  given by
\begin{align}\label{sec2:e13}
\left\{
  \begin{array}{ll}
\displaystyle H_{1k}=\frac{1}{k!}\frac{\partial^{k}}{\partial q^{k}} f_{1}\bigg(s, \sum_{i=0}^{\infty} y_{1i} q^{i}, \sum_{i=0}^{\infty} y_{2i} q^{i}\bigg)\bigg|_{q=0},\vspace{.2cm} \\
\displaystyle H_{2k}=\frac{1}{k!}\frac{\partial^{k}}{\partial q^{k}} f_{2}\bigg(s, \sum_{i=0}^{\infty} y_{1i} q^{i}, \sum_{i=0}^{\infty} y_{2i} q^{i}\bigg)\bigg|_{q=0}.
 \end{array} \right.
\end{align}
Making use of \eqref{sec2:e12},  the $k$th-order deformation equations \eqref{sec2:e10}  are simplified as
\begin{align}\label{sec2:e14}
\left\{
  \begin{array}{ll}
  y_{1k}-\chi_{k}\ y_{1(k-1)}&=\displaystyle c_{10}\bigg[ y_{1(k-1)}-  \bigg(1-\chi_{k}\bigg)\frac{c_1}{a_1}- \int\limits_{0}^{1}G_1(x,s)  p_1(s) H_{1(k-1)}ds\bigg],\vspace{.2cm} \\
y_{2k}-\chi_{k}\ y_{2(k-1)}&= \displaystyle c_{20}\bigg[ y_{2(k-1)}-  \bigg(1-\chi_{k}\bigg)\frac{c_2}{a_2}- \int\limits_{0}^{1}G_2(x,s)  p_2(s) H_{2(k-1)}ds\bigg].
 \end{array} \right.
\end{align}
By taking the initial approximations $y_{10}=\frac{c_1}{a_1}$ and $y_{20}=\frac{c_2}{a_2}$, the solutions components  are computed as
\begin{align}\label{sec2:e15}
\left\{
  \begin{array}{ll}
\displaystyle y_{11}=c_{10}\  y_{10}- c_{10}  \frac{c_1}{a_1}-c_{10}  \int\limits_{0}^{1}G_1(x,s)  p_1(s)  H_{10}  ds,\\
\displaystyle y_{21}=c_{20}\  y_{20}- c_{20}  \frac{c_2}{a_2}-c_{20}  \int\limits_{0}^{1}G_2(x,s)  p_2(s)  H_{20}  ds,\\
\vdots\\
\displaystyle y_{1k} =(1+c_{10})  y_{1(k-1)}-c_{10}  \int\limits_{0}^{1}G_1(x,s)  p_1(s)  H_{1(k-1)}  ds,\\
\displaystyle y_{2k}=(1+c_{20})  y_{2(k-1)}-c_{20}  \int\limits_{0}^{1}G_2(x,s)  p_2(s)  H_{2(k-1)}  ds,~~~ k\geq2.
\end{array} \right.
\end{align}
The $n$th-order approximations to  solutions are obtained  as
\begin{align}\label{sec2:e17}
\left\{
  \begin{array}{ll}
\displaystyle \phi_{1n}(x,c_{10},c_{20})=\sum_{k=0}^{n} y_{1k}(x,c_{10},c_{20}),\\
\displaystyle \phi_{2n}(x,c_{10},c_{20})=\sum_{k=0}^{n} y_{2k}(x,c_{10},c_{20}).
\end{array} \right.
\end{align}
The unknown parameters $c_{10}, c_{20}$ have a significant influence on the convergence of the approximate solution. We next find the optimal values of  $c_{10}, c_{20}$ by solving
\begin{align}\label{sec2:e18}
\frac{\partial E_{1n}(c_{10},c_{20})}{\partial c_{10}}=0,~~~~~~~~\frac{\partial E_{2n}(c_{10},c_{20})}{\partial c_{20}}=0,
\end{align}
where $E_{1n}$ and $E_{2n}$  are given by
\begin{align}\label{sec2:e19}
\left\{
  \begin{array}{ll}
\displaystyle E_{1n}(c_{10},c_{20})=\frac{1}{N}\sum_{k=1}^{N} \bigg(N_1\bigg[\phi_{1n}(x_k,c_{10},c_{20})\bigg]\bigg)^2,\\
\displaystyle E_{2n}(c_{10},c_{20})=\frac{1}{N}\sum_{k=1}^{N} \bigg(N_2 \bigg[\phi_{2n}(x_k,c_{10},c_{20})\bigg]\bigg)^2,~~ x_1, x_2,\ldots,x_N \in [0,1],
\end{array} \right.
\end{align}
Then  these optimal values of  $c_{10}$ and $c_{20}$  will be substituting in  \eqref{sec2:e17} to obtain
the optimal HAM-approximate solution
\begin{align}\label{sec2:e20}
\left\{
  \begin{array}{ll}
\displaystyle \phi_{1n}(x)=\sum_{k=0}^{n} y_{1k}(x),\\
\displaystyle \phi_{2n}(x)=\sum_{k=0}^{n} y_{2k}(x).
\end{array} \right.
\end{align}

\begin{remark}
{\rm  The ADM \cite{duan2015oxygen,singh2020solving}  is a particular case of \eqref{sec2:e15}  when $c_{10}=c_{20}=-1$  and given by
\begin{align}\label{sec2:e21}
\left\{
  \begin{array}{ll}
  \displaystyle  y_{10}=\frac{c_1}{a_1},~~~ y_{20}=\frac{c_2}{a_2},\vspace{0.15cm}\\
\displaystyle y_{1k} =\int\limits_{0}^{1}G_1(x,s) \  p_1(s) \ H_{1(k-1)} ds,\\
\displaystyle y_{2k} =\int\limits_{0}^{1}G_2(x,s) \  p_2(s) \ H_{2(k-1)} ds,~~~~ k\geq1.
\end{array} \right.
\end{align}
Hence, the  ADM-approximate solutions are given by}
\begin{align}\label{sec2:e22}
\left\{
  \begin{array}{ll}
\displaystyle \psi_{1n}(x)=\sum_{k=0}^{n} y_{1k}(x),\\
\displaystyle \psi_{2n}(x)=\sum_{k=0}^{n} y_{2k}(x).
\end{array} \right.
\end{align}
\end{remark}

\section{Convergence analysis}
Let $E= \big(C [0,1], \|y\|\big)$ be a Banach space with norm
 \begin{align}\label{sec4:eq3}
 \|y\|=\displaystyle\max\{\|y_1\|, \|y_2\|\},~~~~~~~y\in E,
 \end{align}
where, $\|y_1\|=\displaystyle\max_{ x\in I=[0,1]} |y_1(x)|$ and $\|y_2\|=\displaystyle\max_{ x\in I} |y_2(x)|$. We next discuss the convergence and error analysis of the proposed method. To do so, we first introduce the vector notations
\begin{align*}
&y=\bigg(\begin{array}{c} y_{1} \\ y_{2} \end{array}\bigg),~~ \textbf{y}_k=\bigg(\begin{array}{c} y_{1k} \\ y_{2k} \end{array}\bigg),~~ c_0=\bigg(\begin{array}{c} c_{10} \\ c_{20} \end{array}\bigg),~~p(x)=\bigg(\begin{array}{c} p_{1}(x) \\ p_{2}(x) \end{array}\bigg),\\
& G(x,s)=\bigg(\begin{array}{c} G_1(x,s) \\ G_2(x,s) \end{array}\bigg),~~H_k=\bigg(\begin{array}{c} H_{1k} \\ H_{2k} \end{array}\bigg),~~ f(x,y_1,y_2)=\bigg(\begin{array}{c} f_1(x,y_1,y_2) \\ f_2(x,y_1,y_2) \end{array}\bigg),
\end{align*}
$\phi_{n}=\left(\begin{array}{c} \phi_{1n} \\ \phi_{2n} \end{array}\right)$ and $M=\max \Bigg(\begin{array}{c} \max_{ x\in I}\big|\int\limits_{0}^{1}G_1(x,s)  p_1(s)  ds\big| \\  \max_{ x\in I}\big|\int\limits_{0}^{1}G_2(x,s)  p_2(s)  ds\big| \end{array}\Bigg).$\\\\

\begin{theorem}\label{sec4:eq4}
Suppose that  the nonlinear function $f(x,z,w)$ satisfy Lipschitz  condition: $|f(x,z,w)-f(x,z^{*},w^{*})|\leq  l_1 |z-z^{*}|+l_2 |w-w^{*}|,$ where $l_1$ and $l_2$ are  Lipschitz constants.  If the parameter $c_0$ is chosen such that there exists constant
$\delta_{c_0} \in (0, 1)$, then the series $\phi_{n}$  defined by \eqref{sec2:e20} is convergent in Banach space $E$.
\end{theorem}
\begin{proof}
From \eqref{sec2:e14} and \eqref{sec2:e17}, we have
\begin{align}\label{sec4:eq1}
\nonumber \phi_{n}&= \sum_{k=0}^{n} \textbf{y}_k(x)=(1+c_0) \sum_{k=1}^{n} \textbf{y}_{k-1}-c_0 \sum_{k=1}^{n}\bigg[\int\limits_{0}^{1} G(x,s)  p(s)  H_{k-1} ds\bigg]\\
&=(1+c_0) \phi_{n-1}-c_0 \int\limits_{0}^{1} G(x,s)  p(s) \sum_{k=1}^{n} H_{k-1} ds.
\end{align}
From  \eqref{sec4:eq1}, we consider for   $n>m$,  for all $n,m\in \mathbb{N}$ as
\begin{align*}
\|\phi_{n}-\phi_{m}\|&= \max_{ x\in I}\bigg| (1+c_0) \big(\phi_{n-1}-\phi_{m-1}\big)-c_0 \int\limits_{0}^{1} G(x,s)  p(s)  \bigg(\sum_{k=1}^{n} H_{k-1}-\sum_{k=1}^{m} H_{k-1}\bigg)ds\bigg|.
 \end{align*}
Using the inequality $\sum_{k=0}^{n} H_{k}\leq f(s,\phi_{1n},\phi_{2n})$ from (\cite{rach2008new}) we have
\begin{align*}
\|\phi_{n}-\phi_{m}\|&\leq  |1+c_0|  \max_{ x\in I}\big| \phi_{n-1}-\phi_{m-1}\big|+ |c_0| \\
&~~~~~~~~~~~~\max_{ x\in I}\bigg|\int\limits_{0}^{1} G(x,s)  p(s) \bigg(f\big(s,\phi_{1(n-1)},\phi_{2(n-1)})-f(s,\phi_{1(m-1)},\phi_{2(m-1)}\big)\bigg)ds\bigg|.
\end{align*}
Applying  the Lipschitz  condition, we get
\begin{align*}
\|\phi_{n}-\phi_{m}\| &\leq  |1+c_0|  \|\phi_{n-1}-\phi_{m-1}\|+ |c_0|  \max_{ x\in I}\bigg|\int\limits_{0}^{1}G(x,s)  p(s)  ds\bigg|
\times \sum _{i=1}^{2} l_i \max_{ t\in I} |\phi_{i(n-1)}-\phi_{i(m-1)}|\\
 &\leq  |1+c_0| \  \|\phi_{n-1}-\phi_{m-1}\|+ |c_0|  \max_{ x\in I}\bigg|\int\limits_{0}^{1}G(x,s)  p(s)  ds\bigg|\\
 &~~~~~~~~~~\times 2L  \max \bigg\{  \bigg\|\phi_{1(n-1)}-\phi_{1(m-1)}\bigg\|,   \bigg\|\phi_{2(n-1)}-\phi_{2(m-1)}\bigg\| \bigg\}\\
&\leq  |1+c_0| \ \bigg\|\phi_{n-1}-\phi_{m-1}\bigg\|+   2LM |c_0| \bigg\|\phi_{n-1}-\phi_{m-1}\bigg\|\\
&\leq  \bigg(|1+c_0| +   2LM |c_0| \bigg) \ \|\phi_{n-1}-\phi_{m-1}\|=\delta_{c_0} \|\phi_{n-1}-\phi_{m-1}\|,
\end{align*}
Thus, we have
\begin{align}\label{sec4:eq6}
\|\phi_{n}-\phi_{m}\|\leq \delta_{c_0} \|\phi_{n-1}-\phi_{m-1}\|.
\end{align}
where $L=\max \{l_1,l_2\},~~\delta_{c_0}=\big(|1+c_0| +   2LM|c_0| \big).$

Setting  $n=m+1$ in \eqref{sec4:eq6}, we have
\begin{align*}
\|\phi_{m+1}-\phi_{m}\|&\leq \delta_{c_0} \|\phi_{m}-\phi_{m-1}\| \leq \delta_{c_0}^2 \|\phi_{m-1}-\phi_{m-2}\|
\leq \ldots \leq \delta_{c_0}^{m} \|\phi_{1}-\phi_{0}\|.
\end{align*}
For all $n,m\in \mathbb{N}$, with  $n>m$, consider
\begin{align*}
\|\phi_{n}-\phi_{m}\|&=\|(\phi_{n}-\phi_{n-1})+(\phi_{n-1}-\phi_{n-1})+\cdots+(\phi_{m+1}-\phi_{m})\|\\
&\leq \|\phi_n-\phi_{n-1}\|+\|\phi_{n-1}-\phi_{n-2}\|+\cdots+\|\phi_{m+1}-\phi_{m}\|\\
&\leq   [\delta_{c_0}^{n-1} + \delta_{c_0}^{n-2} +\cdots+ \delta_{c_0}^{m}] \ \|\phi_{1}-\phi_{0}\|\\
&=  \delta_{c_0}^{m}[ 1+\delta_{c_0}  + \delta_{c_0}^{2} +\cdots+\delta_{c_0}^{n-m-1} ] \ \|\phi_{1}-\phi_{0}\|
=\delta_{c_0}^{m}\left( \frac{1-\delta_{c_0}^{n-m}}{1-\delta_{c_0}}\right)\|\textbf{y}_{1}\|.
\end{align*}
Since $\delta_{c_0}<1$, we have $(1-\delta_{c_0}^{n-m})<1$. It readily follows that
\begin{eqnarray}\label{sec4:eq7}
\|\phi_{n}-\phi_{m}\|\leq&\displaystyle\frac{\delta_{c_0}^{m}}{1-\delta_{c_0}}\|\textbf{y}_{1}\|.
\end{eqnarray}
Letting  $n, m \to \infty$, we obtain
\begin{eqnarray*}
\lim_{n,m\rightarrow \infty} \|\phi_{n}-\phi_{m}\|=0.
\end{eqnarray*}
Therefore,  $\{\phi_{n}\}$ is a Cauchy sequences in the Banach space  $E$.
\end{proof}

\begin{theorem}\label{sec4:eq8}
{\rm If the approximate solution $\phi_{n}$ converges to $\textbf{y}$,
 then the maximum absolute truncated error is given by
 \begin{align}\label{sec4:eq9}
\|\textbf{y}-\phi_{m}\|\leq\frac{\delta_{c_0}^{m} \ M \  |c_0| }{1-\delta_{c_0}}  \max_{x\in I}\big|f(x,y_{10},y_{20}) \big|.
\end{align}}
\end{theorem}
\begin{proof}
Following inequality \eqref{sec4:eq7}, we have
\begin{align*}
\|\phi_{n}-\phi_{m}\|\leq&\displaystyle\frac{\delta_{c_0}^{m}}{1-\delta_{c_0}}\|\textbf{y}_{1}\|.
\end{align*}
For $n\geq m$, as $n\rightarrow \infty$ then   $\phi_{n}\rightarrow \textbf{y}(x)$, the above inequality becomes
\begin{align}\label{sec4:eq10}
\|\textbf{y}(x)-\phi_{m}\|\leq\frac{\delta_{c_0}^{m}}{1-\delta_{c_0}}   \|\textbf{y}_{1}\|.
\end{align}
From \eqref{sec2:e15}, we have $\textbf{y}_{1}=c_0\bigg(\textbf{y}_0-\frac{c}{a}-\int\limits_{0}^{1}G(x,s) \ p(s) \ H_{0} ds\bigg),$
where $\textbf{y}_0=\bigg(\begin{array}{c} y_{10} \\ y_{20} \end{array}\bigg),~~\frac{c}{a}=\bigg(\begin{array}{c} \frac{c_1}{a_1} \\ \frac{c_2}{a_2} \end{array}\bigg),~~ H_{0}=f(x,y_{10},y_{20}).$ Now using $\textbf{y}_0=\frac{c}{a}$, we find
\begin{align}\label{sec4:eq11}
 \|\textbf{y}_{1}\|&= |c_0|\max_{x\in I}\bigg|\frac{c}{a}-\frac{c}{a}-\int\limits_{0}^{1}G(x,s)  p(s) H_{0} ds \bigg|\leq |c_0|M \max_{x\in I}\big|f(x,y_{10},y_{20}) \big|.
\end{align}
Combining \eqref{sec4:eq10} and \eqref{sec4:eq11},  we obtain error estimate as
\begin{align}
\|\textbf{y}-\phi_{m}\|\leq\frac{\delta_{c_0}^{m} \   |c_0|  M }{1-\delta_{c_0}}  \max_{x\in I}\big|f(x,y_{10},y_{20}) \big|.
\end{align}
which completes the proof.
\end{proof}

\begin{remark}\label{sec4:eq12}
\rm{Let us consider the  parameter $c_0$, with  $\delta_{c_0}< 1$, so that
\begin{align*}
|1+c_0| +   2LM  |c_0|< 1~~\Rightarrow~~~ 2LM< \frac{1-|1+c_0| }{ |c_0|},~~c_0\neq 0,
\end{align*}
From the right hand side of this inequality, we have
\begin{align*}
\frac{1-|1+c_0| }{ |c_0|}= \left   \{
  \begin{array}{ll}
  -1-\displaystyle\frac{2}{c_0},~~~~ & \hbox{$c_0<-1$}, \\
  1,~~~~ & \hbox{$c_0 \in [-1,0)$}\\
  -1,~~~~ & \hbox{$c_0>0$}\\
\end{array}
\right.
\end{align*}
We can choose value of parameter $c_0\in [-1,0)$.}
\end{remark}

\section{Numerical results and discussion}
 In order to compare our numerical results with the results obtained by the ADM  \cite{duan2015oxygen}, we define the absolute residual error as
\begin{align}
Res_{1n}=\bigg|\frac{1}{ p_1(x) }\bigg(p_{1}(x)\phi_{1n}'(x)\bigg)'-f_1\big(x,\phi_{1n}(x),\phi_{2n}(x)\big)\bigg|,\\
Res_{2n}=\bigg|\frac{1}{ p_2(x) }\bigg(p_{2}(x)\phi_{2n}'(x)\bigg)'-f_2\big(x,\phi_{1n}(x),\phi_{2n}(x)\big)\bigg|,
\end{align}
\begin{align}
res_{1n}=\bigg|\frac{1}{ p_1(x) }\bigg(p_{1}(x)\psi_{1n}'(x)\bigg)'-f_1\big(x,\psi_{1n}(x),\psi_{2n}(x)\big)\bigg|,\\
res_{2n}=\bigg|\frac{1}{ p_2(x) }\bigg(p_{2}(x)\psi_{2n}'(x)\bigg)'-f_2\big(x,\psi_{1n}(x),\psi_{2n}(x)\big)\bigg|,
\end{align}
where $(\phi_{1n}(x), \phi_{2n}(x))$ are the HAM approximate solutions and $(\psi_{1n}(x), \psi_{2n}(x))$ are the ADM approximate solutions.

In Tables \ref{tab1}-- \ref{tab8}, we compare the numerical results of  $\big(\phi_{1n}(x), \phi_{2n}(x)\big)$, $\big(\psi_{1n}(x), \psi_{2n}(x)\big)$,   $\big(Res_{1n}, Res_{2n}\big)$, and $\big(res_{1n}, res_{2n}\big)$, obtained by the HAM and the ADM \cite{duan2015oxygen}. From the comparison, we observe that the convergence of the proposed method is very fast and the only the first few terms of the series provide an excellent approximation to the exact solution. The optimal choice of the convergence control parameters permits to obtain an excellent approximation.

\begin{example}\label{sec5:eq1}
{\rm Consider the system of Lane-Emden-Fowler equations that relates the concentration of the carbon substrate
and the concentration of oxygen \cite{rach2014solving,singh2020solving,wazwaz2014study}}
\begin{align}
\left\{
  \begin{array}{ll}
 \displaystyle  (x^{k_1}y'_{1})' = x^{k_1}\bigg(-b+\frac{a\ y_1 \ y_2 }{(l_{1}+y_1)(m_{1}+y_2)}+\frac{c\ y_1\ y_2 }{(l_{2}+y_1)(m_{2}+y_2)}\bigg),~x\in(0,1),\vspace{0.1cm}\\
 \displaystyle  (x^{k_2}y'_{2})' =x^{k_2}\bigg(\frac{d\  y_1 \ y_2 }{(l_{1}+y_1)(m_{1}+y_2)}+ \frac{e\ y_1\ y_2 }{(l_{2}+y_1)(m_{2}+y_2)}\bigg),\vspace{0.1cm}\\
y'_{1}(0)=0,~~y'_{2}(0)=0 ~~~~y_{1}(1)=1,~~~~y_{2}(1)=1,
\end{array} \right.
\end{align}
\end{example}
where  the functions $y_1(x)$ and $y_2(x)$ are the concentration of the carbon substrate and the concentration of oxygen, respectively.  We fix the parameters $l_1=l_2=m_1=m_2= \frac{1}{10000}, a=5,b=1,c=d= \frac{1}{10}, e=\frac{5}{100}$ and $a_1=a_2=1, b_1=b_2=0, c_1=c_2=1$. Applying  \eqref{sec2:e15} with $y_{10}=y_{20}=1$ and $k_1=k_2=1$, we find the HAM-approximations as
\begin{align*}
&\phi_{13}(x,c_{10},c_{20})=1-1.9995 c_{10}-0.9996 c_{10}^2+3.51\times 10^{-6} c_{10} c_{20}+(1.9995 c_{10}+0.999 c_{10}^2\\
&~~~~~~~-4.69\times 10^{-6} c_{10} c_{20}) x^2+(3.12\times 10^{-5} c_{10}^2+1.17\times10^{-6} c_{10} c_{20}) x^4.\\
&\phi_{23}(x,c_{10},c_{20})=1-0.074985 c_{20}+2.81\times10^{-6} c_{10} c_{20}-0.03749 c_{20}^2+(0.07498 c_{20}\\
&~~~~~~~~~~~~~-3.75\times10^{-6} c_{10} c_{20}+0.0374 c_{20}^2) x^2+(9.37\times10^{-7} c_{10} c_{20}+3.51\times10^{-8} c_{20}^2) x^4.
\end{align*}
Using  the formula \eqref{sec2:e18}, we find $c_{10}=-1.00010501, c_{20}=-1.0000443$, and  hence the HAM-approximations
\begin{align*}
\left\{
  \begin{array}{ll}
\phi_{13}(x)=1.99985-0.99988 x^2+3.24108\times 10^{-5} x^4.\\
\phi_{23}(x)=1.0375-0.0374964 x^2+9.72266\times 10^{-7} x^4.
\end{array} \right.
\end{align*}

Using  \eqref{sec2:e15} with  $y_{10}=y_{20}=1$ and $k_1=k_2=2$,  we find  the HAM-approximations as
\begin{align*}
&\phi_{13}(x,c_{10},c_{20})=1-1.333 c_{10}-0.666461 c_{10}^2+1.45\times10^{-6} c_{10} c_{20}+(1.333 c_{10}+0.6664 c_{10}^2\\
&~~~~~~~~~~-2.08\times10^{-6} c_{10} c_{20}) x^2+(0.00001665 c_{10}^2+6.24\times10^{-7} c_{10} c_{20}) x^4.\\
&\phi_{23}(x,c_{10},c_{20})=1-0.04999  c_{20}+1.16\times10^{-6} c_{10} c_{20}-0.024995 c_{20}^2+(0.04999 c_{20}\\
&~~~~~~~~~~~~-1.66\times10^{-6} c_{10}  c_{20}+0.02499 c_{20}^2) x^2+(4.99\times10^{-7} c_{10}  c_{20}+1.87\times10^{-8}  c_{20}^2) x^4.
\end{align*}
Applying  \eqref{sec2:e18}, we find $c_{10}=-0.995713, c_{20}=-0.996167$, and obtain the HAM-approximate solutions
\begin{align*}
\left\{
  \begin{array}{ll}
\phi_{13}(x)=1.66653-0.666545 x^2+1.713 \times 10^{-5} x^4.\\
\phi_{23}(x)=1.025-0.0249963 x^2+5.142\times 10^{-7} x^4.
\end{array} \right.
\end{align*}

\begin{table}[htbp]
\caption{Comparison of numerical results for $k_1=k_2=1$ of  Example \ref{sec5:eq1}}
\centering
\setlength{\tabcolsep}{0.05in}
\begin{tabular}{l|cc cc|cccc}
\hline
&& Solutions  &&    &&  Errors &\\
\cline{1-9}
$x$ & $\phi_{13}$ & $\psi_{13}$\cite{singh2020solving} & $\phi_{23}$ & $\psi_{23}$\cite{singh2020solving} &$Res_{13}$ &$res_{13}$\cite{singh2020solving} &$Res_{23}$ &$res_{23}$\cite{singh2020solving}\\
\cline{1-9}
0.1	&	1.9898484	&	1.9898484	&	1.0371204	&	1.0371204	&	2.46E-04	&	2.46E-04	&	7.40E-06	&	7.40E-06	\\
0.3	&	1.9098583	&	1.9098583	&	1.0341207	&	1.0341207	&	2.17E-04	&	2.17E-04	&	6.51E-06	&	6.51E-06	\\
0.5	&	1.7498793	&	1.7498793	&	1.0281213	&	1.0281213	&	1.61E-04	&	1.60E-04	&	4.83E-06	&	4.82E-06	\\
0.7	&	1.5099140	&	1.5099140	&	1.0191224	&	1.0191224	&	8.62E-05	&	8.62E-05	&	2.58E-06	&	2.58E-06	\\
0.9	&	1.1899659	&	1.1899659	&	1.0071239	&	1.0071239	&	1.51E-05	&	1.51E-05	&	4.55E-07	&	4.55E-07	\\
\hline
\end{tabular}
\label{tab1}
\end{table}
\begin{table}[htbp]
\caption{Comparison of numerical  results  for $k_1=k_2=2$ of  Example \ref{sec5:eq1}}
\centering
\setlength{\tabcolsep}{0.05in}
\begin{tabular}{l|cc cc|cccc}
\hline
&& Solutions  &&    &&  Errors &\\
\cline{1-9}
$x$ & $\phi_{13}$ & $\psi_{13}$\cite{duan2015oxygen}  & $\phi_{23}$ & $\psi_{23}$\cite{duan2015oxygen} &$Res_{13}$ &$res_{13}$\cite{duan2015oxygen} &$Res_{23}$ &$res_{23}$\cite{duan2015oxygen} \\
\cline{1-9}
0.1	&	1.6598623	&	1.6598747	&	1.0247458	&	1.0247462	&	5.49E-05	&	1.31E-04    &	1.65E-06	&	3.94E-06	\\
0.3	&	1.6065388	&	1.6065503	&	1.0227461	&	1.0227465	&	3.85E-05	&	1.14E-04	&	1.16E-06	&	3.44E-06	\\
0.5	&	1.4998926	&	1.4999020	&	1.0187467	&	1.0187470	&	7.73E-06	&	8.34E-05	&	2.37E-07	&	2.50E-06	\\
0.7	&	1.3399248	&	1.3399312	&	1.0127477	&	1.0127479	&	3.18E-05	&	4.31E-05	&	9.50E-07	&	1.29E-06	\\
0.9	&	1.1266376	&	1.1266400	&	1.0047491	&	1.0047492	&	6.69E-05	&	7.12E-06	&	2.00E-06	&	2.13E-07	\\
\hline
\end{tabular}
\label{tab2}
\end{table}

\begin{example}\label{sec5:eq2}
{\rm Consider the system of Lane-Emden-Fowler equation  occurs in catalytic diffusion reactions \cite{flockerzi2011coupled,rach2014solving}}
\begin{align}
\left\{
  \begin{array}{ll}
 \displaystyle   (x^{2}y'_{1})'=x^{2} (a\ y_{1}^{2}+b\ y_1 y_2),~~~~~x\in(0,1)\vspace{0.15cm}\\
 \displaystyle   (x^{2}y'_{2})'=x^{2}(c\ y_{1}^{2}-d\ y_1 y_2),\vspace{0.15cm}\\
 y'_{1}(0)=0,~~y'_{2}(0)=0,~~~y_{1}(1)=1,~~  y_{2}(1)=2.
\end{array} \right.
\end{align}
\end{example}
Here,  $k_1=k_2=2$, $a_1=a_2=1$, $b_1=b_2=0$, $c_1=1$  and $c_2=2$. Using  \eqref{sec2:e15} with  $y_{10}=1$, $y_{20}=2$ and $a=1$, $b=\frac{2}{5}$, $c= \frac{1}{2}$, $d=1$, we  find the HAM approximate solution  as
\begin{align*}
&\phi_{14}(x,c_{10},c_{20})=1+\frac{9 c_{10}}{10}+\frac{597 c_{10}^2}{500}+\frac{55973 c_{10}^3}{105000}+\frac{7 c_{10} c_{20}}{120}+\frac{1613 c_{10}^2 c_{20}}{47250}+\frac{65 c_{10} c_{20}^2}{3024}\\
&+\bigg(-\frac{9 c_{10}}{10}-\frac{33 c_{10}^2}{25}-\frac{9611 c_{10}^3}{15000}-\frac{c_{10} c_{20}}{12}-\frac{1409 c_{10}^2 c_{20}}{27000}-\frac{67 c_{10} c_{20}^2}{2160}\bigg) x^2
+\bigg(\frac{63 c_{10}^2}{500}+\frac{563 c_{10}^3}{5000}\\
&+\frac{c_{10} c_{20}}{40}+\frac{91 c_{10}^2 c_{20}}{4500}+\frac{7 c_{10} c_{20}^2}{720}\bigg) x^4
+\bigg(-\frac{173 c_{10}^3}{35000}-\frac{137 c_{10}^2 c_{20}}{63000}-\frac{c_{10} c_{20}^2}{5040}\bigg) x^6.\\
&\phi_{24}(x,c_{10},c_{20})=2+\frac{5 c_{20}}{4}+\frac{63 c_{10} c_{20}}{200}+\frac{1961 c_{10}^2 c_{20}}{14000}+\frac{67 c_{20}^2}{48}+\frac{673 c_{10} c_{20}^2}{5040}+\frac{3139 c_{20}^3}{6048}\\
&+\bigg(-\frac{5 c_{20}}{4}-\frac{9 c_{10} c_{20}}{20}-\frac{413 c_{10}^2 c_{20}}{2000}-\frac{35 c_{20}^2}{24}-\frac{713 c_{10} c_{20}^2}{3600}-\frac{487 c_{20}^3}{864}\bigg) x^2
+\bigg(\frac{27 c_{10} c_{20}}{200}+\frac{141 c_{10}^2 c_{20}}{2000}\\
&+\frac{c_{20}^2}{16}+\frac{83 c_{10} c_{20}^2}{1200}+\frac{13 c_{20}^3}{288}\bigg) x^4+\bigg(-\frac{57 c_{10}^2 c_{20}}{14000}-\frac{13 c_{10} c_{20}^2}{2800}-\frac{c_{20}^3}{2016}\bigg) x^6.
\end{align*}
Using  \eqref{sec2:e18}, we find $c_{10}=-0.767463, c_{20}=-0.789762$. Hence, the HAM approximate solution
\begin{align*}
\left\{
  \begin{array}{ll}
\phi_{14}=0.780767+0.191485 x^2+0.0244069 x^4+0.00334088 x^6.\\
\phi_{24}=1.68960+0.273372 x^2+0.0326694 x^4+0.00436072 x^6.
\end{array} \right.
\end{align*}
Using  \eqref{sec2:e15} with $a=b=c=d=1$ and \eqref{sec2:e18} with $c_{10}=-0.689796, c_{20}=-0.708697$, and obtain the HAM-approximations as
\begin{align*}
\left\{
  \begin{array}{ll}
\phi_{14}=0.674423+0.271204 x^2+0.0454739 x^4+0.00889876 x^6.\\
\phi_{24}=1.67352+0.27178 x^2+0.0455586 x^4+0.00914259 x^6.
\end{array} \right.
\end{align*}

\begin{table}[htbp]
\caption{Comparison of numerical  results  for $a=1$, $b=\frac{2}{5}$, $c= \frac{1}{2}$, $d=1$ of  Example \ref{sec5:eq2}}
\centering
\setlength{\tabcolsep}{0.04in}
\begin{tabular}{l|cc| cc|cc|cc}
\hline
&& Solutions  &&    &&  Errors &\\
\cline{1-9}
$x$ & $\phi_{14}$ & $\psi_{14}$\cite{duan2015oxygen} & $\phi_{24}$ & $\psi_{24}$\cite{duan2015oxygen}  &$Res_{14}$ &$res_{14}$\cite{duan2015oxygen}  &$Res_{24}$ &$res_{24}$\cite{duan2015oxygen}\\
\cline{1-9}
0.1	&	0.7826843&	0.7658317	&	1.6923350	&	1.6713156	&	1.13E-02  &	2.26E-01   &	1.59E-02	&	7.63E-01	\\
0.3	&	0.7982008&	0.7835530	&	1.7144693	&	1.6962143	&	9.45E-03  &	1.95E-01   &	1.34E-02	&	7.54E-01	\\
0.5	&	0.8302159&	0.8194185	&	1.7600510	&	1.7466228	&	5.96E-03  &	1.41E-01   &	9.17E-03     &	7.21E-01	\\
0.7	&	0.8808479&	0.8746115	&	1.8319072	&	1.8241881	&	4.43E-04  &	7.71E-02   &	2.78E-03 	&	6.85E-01	\\
0.9	&	0.9536588&	0.9517495	&	1.9347811	&	1.9324416	&	1.11E-02  &	2.06E-02   &	1.02E-02	&   6.77E-01	\\
\hline
\end{tabular}
\label{tab3}
\end{table}

\begin{table}[htbp]
\caption{Comparison of numerical  results  for  $a=b=c=d=1$ of  Example \ref{sec5:eq2}}
\centering
\setlength{\tabcolsep}{0.05in}
\begin{tabular}{l|cc cc|cccc}
\hline
&& Solutions  &&    &&  Errors &\\
\cline{1-9}
$x$ & $\phi_{14}$ & $\psi_{14}$\cite{duan2015oxygen} & $\phi_{24}$ & $\psi_{24}$\cite{duan2015oxygen} &$Res_{14}$ &$res_{14}$\cite{duan2015oxygen} &$Res_{24}$ &$res_{24}$\cite{duan2015oxygen} \\
\cline{1-9}
0.1	&	0.6771397   &	0.5967530	&	1.6762408	&	1.5967530	&	4.27E-02	&	1.143631&	4.62E-02 	&	1.1436	\\
0.3	&	0.6992063	&	0.6293170	&	1.6983544	&	1.6293170	&	3.57E-02	&	0.99354	&	3.94E-02	&	0.9935	\\
0.5	&	0.7452053	&	0.6936848	&	1.7444538	&	1.6936848	&	2.26E-02	&	0.72673	&	2.71E-02	&	0.7267	\\
0.7	&	0.8192784	&	0.7895710	&	1.8187051	&	1.7895710	&	1.36E-03	&	0.40619	&	8.10E-03	&	0.4061	\\
0.9	&	0.9286631	&	0.9196325	&	1.9284103	&	1.9196325	&	4.41E-02    &	0.11286	&	3.25E-02	&	0.1128	\\
\hline
\end{tabular}
\label{tab4}
\end{table}

\begin{example}\label{sec5:eq3}
{\rm Consider the following system of Lane-Emden-Fowler with boundary conditions}
\begin{align}
\left\{
  \begin{array}{ll}
 \displaystyle  (x^{3}y'_{1})'=x^{3} (y_{1}y_{2} +7+(y_{1}-1)^2),~~~~~~~x\in(0,1),\vspace{0.1cm}\\
 \displaystyle  (x^{4}y'_{2})'=x^{4}(y_{1} y_{2} -11+(y_{2}-1)^2),\vspace{0.1cm}\\
y'_{1}(0)=0,~~y'_{2}(0)=0,~~y_{1}(1)=2,~~y_{2}(1)=0.
\end{array} \right.
\end{align}
\end{example}
The exact solutions are $y_1(x)=3-x^2$ and $y_2(x)=-1+x^2.$ Here, $a_1=a_2=1$, $b_1=b_2=0$ and $c_1=2$ and $c_2=0$. Using \eqref{sec2:e15} with  $y_{10}=2$ and  $y_{20}=0$, we find  the series solutions as
\begin{align*}
&\phi_{13}(x,c_{10},c_{20})=2-3 c_{10}-\frac{5 c_{10}^2}{2}-\frac{217 c_{10}^3}{288}-\frac{c_{10} c_{20}}{2}-\frac{23 c_{10}^2 c_{20}}{288}-\frac{c_{10} c_{20}^2}{6}
+\bigg(3 c_{10}+\frac{9 c_{10}^2}{4}\\
&+\frac{2 c_{10}^3}{3}+\frac{3 c_{10} c_{20}}{4}+\frac{c_{10}^2 c_{20}}{12}+\frac{c_{10} c_{20}^2}{4}\bigg) x^2+\left(\frac{c_{10}^2}{4}+\frac{c_{10}^3}{16}-\frac{c_{10} c_{20}}{4}+\frac{c_{10}^2 c_{20}}{48}-\frac{c_{10} c_{20}^2}{12}\right) x^4\\
&~~~~~~~~+\left(\frac{7 c_{10}^3}{288}-\frac{7 c_{10}^2 c_{20}}{288}\right) x^6.\\
&\phi_{23}(x,c_{10},c_{20})=3 c_{20}+3 c_{20}^2+\frac{89 c_{10} c_{20}^2}{1890}+\frac{1801 c_{20}^3}{1890}+\left(-3 c_{20}-3 c_{20}^2-\frac{c_{10} c_{20}^2}{10}-\frac{9 c_{20}^3}{10}\right)x^2\\
&~~~~~~~~+\left(\frac{c_{10} c_{20}^2}{14}-\frac{c_{20}^3}{14}\right) x^4+\left(-\frac{c_{10} c_{20}^2}{54}+\frac{c_{20}^3}{54}\right) x^6.
\end{align*}
Applying  \eqref{sec2:e18}, we find $c_{10}=c_{20}=-1$, and obtain the exact solutions
\begin{align*}
\phi_{13}(x)=3-x^2,~~~~~~~\phi_{23}(x)=-1+x^2.
\end{align*}

\begin{example}\label{sec5:eq4}
{\rm Consider the following system of Lane-Emden-Fowler with boundary conditions}
\begin{align}
\left\{
  \begin{array}{ll}
 \displaystyle  (x^{5}y'_{1})'= x^{5}(-8e^{y_{1}}-16e^{\frac{-y_{2}}{2}}),~~~~~~~x\in(0,1),\vspace{0.1cm}\\
 \displaystyle  (x^{3}y'_{1})'= x^{3} (8 e^{-y_{2}}+8e^{\frac{y_{1}}{2}}),\vspace{0.1cm}\\
y'_{1}(0)=0,~~y'_{2}(0)=0,~~y_{1}(1)=-2\ln2,~~y_{2}(1)=2\ln2.
\end{array} \right.
\end{align}
\end{example}
The exact solutions are $y_1(x)=-2 \ln(1+x^2)$ and $y_2(x)=2\ln(1+x^2).$ Here, $k_1=5$, $k_2=3$, $a_1=a_2=1$, $b_1=b_2=0$, $c_1=-2\ln2$ and  $c_2=2\ln2$. Making use of \eqref{sec2:e15} and  \eqref{sec2:e18} with  $y_{10}=-2\ln2$, $y_{20}=2\ln2$, we obtain
the HAM-approximations  with ($c_{10}= -0.763735, c_{20}=-0.743226$) as
\begin{align*}
\left\{
  \begin{array}{ll}
&\phi_{15}(x)=-2.05166+0.587014 x^2+0.0661967 x^4+0.00974409 x^6+0.00215741 x^8\\
&~~~~~~+0.000103467 x^{10}+0.000153951 x^{12}.\\
&\phi_{25}(x)=1.95556-0.499767 x^2-0.0589602 x^4-0.00851556 x^6-0.00181135 x^8\\
&~~~~~~~-0.000100184 x^{10}-0.000115625 x^{12}.
\end{array} \right.
\end{align*}

\begin{table}[htbp]
\caption{Comparison of numerical  results  for $k_1=5$,  $k_2=3$ of  Example \ref{sec5:eq4}}
\centering
\setlength{\tabcolsep}{0.05in}
\begin{tabular}{l|cc cc|cccc}
\hline
&& Solutions  &&    &&  Errors &\\
\cline{1-9}
$x$ & $\phi_{15}$ & $\psi_{15}$\cite{duan2015oxygen}  & $\phi_{25}$ & $\psi_{25}$\cite{duan2015oxygen} &$Res_{25}$ &$res_{25}$\cite{duan2015oxygen} &$Res_{25}$ &$res_{25}$\cite{duan2015oxygen} \\
\cline{1-9}
0.1	&	-2.0457870	&	-2.0358737	&	1.9505604	&	1.9379913	&	2.20E-03   &	0.398753	&	1.65E-03	&	0.290684	\\
0.3	&	-1.9982891	&	-1.9904854	&	1.9101010	&	1.8999319	&	1.51E-03	&	0.302292	&	1.17E-03	&	0.219317	\\
0.5	&	-1.9006122	&	-1.8958769	&	1.8267970   &	1.8202489	&	7.51E-04	&	0.168370	&	6.34E-04	&	0.120416	\\
0.7	&	-1.7468574	&	-1.7447898	&	1.6954112	&	1.6922536	&	3.74E-04	&	6.38E-02   &	3.94E-04	&	4.39E-02\\
0.9	&	-1.5265642	&	-1.5261201	&	1.5066963	&	1.5059088	&	6.69E-04	&	1.18E-02	&	7.42E-04	&	7.51E-03	\\
\hline
\end{tabular}
\label{tab5}
\end{table}


\begin{example}\label{sec5:eq5}
{\rm  Consider the following system of  Lane-Emden-Fowler with boundary conditions}
\begin{align}
\left\{
  \begin{array}{ll}
 \displaystyle  (x^{2}y'_{1})'=x^{2}(2 \left(7+e^{y_{2}} \right)e^{-2y_1 }),~~~~~~~x\in(0,1),\vspace{0.1cm}\\
 \displaystyle  (x^{2}y'_{2})'=x^{2}(2 \left(11+e^{y_{1}} \right)e^{-2y_{2}}),\vspace{0.1cm}\\
y'_{1}(0)=0,~~y'_{2}(0)=0,~~~y_{1}(1)=\ln 4,~~y_{2}(1)=\ln 5.
\end{array} \right.
\end{align}
\end{example}
The exact solutions are $y_1(x)=\ln(4+x^2)$  and $y_2(x)=\ln(5+x^2).$ Here,  $k_1=k_2=2$, $a_1=a_2=1$, $b_1=b_2=0$, $c_1=\ln 4$ and  $c_2=\ln 5$. Using \eqref{sec2:e15} and  \eqref{sec2:e18} with  $y_{10}=\ln 4$, $y_{20}=\ln 5$, we find  the HAM-approximate series solutions with
($c_{10}=-0.766209, c_{20}= -0.800994$) as
 \begin{align*}
 \left\{
  \begin{array}{ll}
\phi_{14}=1.58464-0.18032 x^2-0.0160564 x^4-0.00163913 x^6-0.000328184 x^8.\\
\phi_{24}=1.77423-0.152098 x^2-0.0114857 x^4-0.00103977 x^6-0.00016868 x^8.
\end{array} \right.
\end{align*}
\begin{table}[htbp]
\caption{Comparison of numerical  results  for $k_1=k_2=2$ of  Example \ref{sec5:eq5}}
\centering
\setlength{\tabcolsep}{0.05in}
\begin{tabular}{l|cc cc|cccc}
\hline
&& Solutions  &&    &&  Errors &\\
\cline{1-9}
$x$ & $\phi_{14}$ & $\psi_{14}$\cite{duan2015oxygen}  & $\phi_{24}$ & $\psi_{24}$\cite{duan2015oxygen} &$Res_{24}$ &$res_{24}$\cite{duan2015oxygen} &$Res_{24}$ &$res_{24}$\cite{duan2015oxygen} \\
\cline{1-9}
0.1	&	1.5828329	&	1.5769131	&	1.7727080	&	1.7709758	&	2.14E-03	&	8.30E-02	&	9.45E-04	&	2.33E-02	\\
0.3	&	1.5682776	&	1.5632335	&	1.7604475	&	1.7589804	&	1.77E-03	&	6.88E-02	&	7.83E-04	&	1.90E-02	\\
0.5	&	1.5385273	&	1.5349546	&	1.7354708	&	1.7344446	&	1.31E-03	&	4.57E-02	&	5.90E-04	&	1.21E-02	\\
0.7	&	1.4922141	&	1.4902763	&	1.6968123	&	1.6962663	&	1.23E-03	&	2.21E-02	&	5.60E-04	&	5.48E-03	\\
0.9	&	1.4270318	&	1.4264972	&	1.6428697	&	1.6427234	&	2.46E-03	&	5.17E-03	&	1.06E-03	&	1.12E-03	\\
\hline
\end{tabular}
\label{tab6}
\end{table}

\begin{example}\label{sec5:eq6}
{\rm Consider the following system of Lane-Emden-Fowler with boundary conditions}
\begin{align}
\left\{
  \begin{array}{ll}
 \displaystyle  (x^{2}y'_{1})'= x^{2}\bigg(-6 \left(e^{\frac{y_2}{3}}+4\right)e^{\frac{2y_1}{3}}\bigg),~~~~~~~x\in(0,1),\vspace{0.1cm}\\
 \displaystyle (x^{2}y'_{2})'=x^{2}\bigg(6 \left(e^{\frac{-y_1}{3}}+4\right)e^{\frac{-2y_2}{3}}\bigg),\vspace{0.1cm}\\
y'_{1}(0)=0,~~y'_{2}(0)=0,~~~y_{1}(1)=-3\ln 3,~~ y_{2}(1)=3\ln 3.
\end{array} \right.
\end{align}
\end{example}
The exact solutions are $y_1=-3 \ln(2+x^2),~~y_2=3\ln(2+x^2).$ Here, $k_1=k_2=2$, $a_1=a_2=1$, $b_1=b_2=0$, $c_1=-3\ln 3$ and $c_2=3\ln 3$. In view of \eqref{sec2:e15}  and  \eqref{sec2:e18} with  $y_{10}=-3\ln 3$, $y_{20}=3\ln 3$, we find  the HAM-approximate series solutions  with ($c_{10}=-0.764679, c_{20}=-0.764679$)  as
\begin{align*}
\left\{
  \begin{array}{ll}
\phi_{14}(x)=-3.91525+0.564146 x^2+0.049265 x^4+0.00500793 x^6+0.00099786 x^8.\\
\phi_{24}(x)=3.91525-0.564146 x^2-0.049265 x^4-0.00500793 x^6-0.00099786 x^8.
\end{array} \right.
\end{align*}
\begin{table}[htbp]
\caption{Comparison of numerical  results   for $k_1=2$, $k_2=2$ of  Example \ref{sec5:eq6}}
\centering
\setlength{\tabcolsep}{0.03in}
\begin{tabular}{l|cc cc|cccc}
\hline
&& Solutions  &&    &&  Errors &\\
\cline{1-9}
$t$ & $\phi_{14}$ & $\psi_{14}$\cite{duan2015oxygen}  & $\phi_{24}$ & $\psi_{24}$\cite{duan2015oxygen} &$Res_{24}$ &$res_{24}$\cite{duan2015oxygen} &$Res_{24}$ &$res_{24}$\cite{duan2015oxygen} \\
\cline{1-9}
0.1	&	-3.9096075	&	-3.8933979	&	3.9096075	&	3.8933979	&	6.39E-03	&	0.225756	&	6.39E-03	&	2.25E-02	\\
0.3	&	-3.8640780	&	-3.8502979	&	3.8640780	&	3.8502979	&	5.34E-03	&	0.186062	&	5.34E-03	&	5.58E-02	\\
0.5	&	-3.7710561	&	-3.7613414	&	3.7710561	&	3.7613414	&	4.11E-03	&	0.122139	&	4.11E-03	&	6.10E-02	\\
0.7	&	-3.6263470	&	-3.6211160	&	3.6263470	&	3.6211160	&	4.05E-03	&	5.78E-02	&	4.05E-03	&	4.04E-02	\\
0.9	&	-3.4228817	&	-3.4214542	&	3.4228817	&	3.4214542	&	8.02E-03	&	1.30E-02	&	8.02E-03	&	1.17E-02	\\
\hline
\end{tabular}
\label{tab7}
\end{table}

\begin{example}\label{sec5:eq7}
{\rm Consider the following system of Lane-Emden-Fowler with boundary conditions}
\begin{align}
\left\{
  \begin{array}{ll}
 \displaystyle  (x^{3}y'_{1})'= x^{3}(-\left(3+y_2^2\right)y_1^5),~~~~~~~x\in(0,1),\vspace{0.1cm}\\
 \displaystyle  (x^{4}y'_{1})'= x^{4}(\left(4y_1^{-2}+1\right)y_2^{-3}),\vspace{0.1cm}\\
y'_{1}(0)=0,~~y'_{2}(0)=0,~~~y_{1}(1)=\frac{1}{\sqrt{2}},~~~~y_{2}(1)=\sqrt{2}.
\end{array} \right.
\end{align}
\end{example}
The exact solutions are $y_1(x)=\frac{1}{\sqrt{1+x^2}}$ and $y_2(x)=\sqrt{1+x^2}.$ Here, $a_1=a_2=1$, $b_1=b_2=0$, $c_1=\frac{1}{\sqrt{2}}$ and $c_2=\sqrt{2}$. In view of \eqref{sec2:e15} and \eqref{sec2:e18}  with $y_{10}=\frac{1}{\sqrt{2}}$, $y_{20}=\sqrt{2}$, we find  the HAM-approximate series solutions with  ($c_{10}=-0.718977, c_{20}=-0.726659$), as
\begin{align*}
\left\{
  \begin{array}{ll}
\phi_{14}(x)=0.626038+0.0695736 x^2+0.0099351 x^4+0.00129509 x^6+0.000264696 x^8.\\
\phi_{24}(x)=1.67507-0.237503 x^2-0.0198153 x^4-0.00227264 x^6-0.0012686 x^8.
\end{array} \right.
\end{align*}

\begin{table}[htbp]
\caption{Comparison of numerical  results   for $k_1=3$, $k_2=4$ of  Example \ref{sec5:eq7}}
\centering
\setlength{\tabcolsep}{0.03in}
\begin{tabular}{l|cc cc|cccc}
\hline
&& Solutions  &&    &&  Errors &\\
\cline{1-9}
$t$ & $\phi_{14}$ & $\psi_{14}$\cite{duan2015oxygen}  & $\phi_{24}$ & $\psi_{24}$\cite{duan2015oxygen} &$Res_{24}$ &$res_{24}$\cite{duan2015oxygen} &$Res_{24}$ &$res_{24}$\cite{duan2015oxygen} \\
\cline{1-9}
0.1	&	0.6267350	&	0.6326026	&	1.6726961	&	1.6660461	&	1.67E-03	&	0.123810	&	8.99E-03	&	0.195077	\\
0.3	&	0.6323813	&	0.6373236	&	1.6535356	&	1.6483411	&	1.35E-03	&	0.103217	&	7.53E-03	&	0.146187	\\
0.5	&	0.6440739	&	0.6474853	&	1.6144184	&	1.6113861	&	9.66E-04	&	0.069423	&	6.11E-03	&   7.46E-02	\\
0.7	&	0.6626824	&	0.6644484	&	1.5535984	&	1.5524363	&	1.01E-03	&	0.034154	&	6.20E-03	&	1.78E-02	\\
0.9	&	0.6897135	&	0.6901626	&	1.4679409	&	1.4677736	&	2.79E-03	&	0.008057	&	1.01E-02	&	1.44E-03	\\
\hline
\end{tabular}
\label{tab8}
\end{table}

\section{Concluding remarks}
The theory of a system of  singular differential equations finds its vital presence in many of the natural or physical processes such as   catalytic diffusion reactions \cite{rach2014solving}, and some system of  Lane-Emden equations that relate the concentration of the carbon substrate and the concentration of oxygen \cite{duan2015oxygen}. In this paper, a general analytical approach has been presented for the approximate series solution of system of singular differential equations with Neumann and Robin type boundary conditions. Unlike the Adomian decomposition method, the proposed technique does not require the computation of unknown constants, and it contains convergence parameters which ensure the fast convergence of the series solution. Unlike, the other numerical methods,   our approach does not require any linearization or discretization of variables. Numerical results obtained by the present method are better than the results obtained by the Adomian decomposition method reported in  \cite{duan2015oxygen}, as shown in Tables \ref{tab1}-\ref{tab8}.   Convergence and error estimation of the homotopy analysis method for a system of singular boundary value problems is provided under quite general conditions.


\end{document}